\documentclass[12pt]{amsart}
\usepackage{amsmath, amssymb, amsthm, mathtools, relsize, paralist, hyperref, setspace}
\usepackage[paper=a4paper,left=2.9cm,right=2.9cm,top=3cm,bottom=3cm]{geometry}

\newtheorem{theorem}{Theorem}[section]

\theoremstyle{definition}

\theoremstyle{remark}

\numberwithin{equation}{section}

\newcommand{ \prost }{ \operatorname{p} }
\newcommand{ \pprost }{ \operatorname{\tilde{p}} }
\newcommand{ \bs }[1]{ \boldsymbol{#1} }

\begin{document}

\title[Lattice Packings of Cross-polytopes]
      {Lattice Packings of Cross-polytopes from {R}eed--{S}olomon Codes and {S}idon Sets}

\author{Mladen Kova\v{c}evi\'c}

\thanks{The author is with the Faculty of Technical Sciences,
University of Novi Sad, 21000 Novi Sad, Serbia
(email: kmladen@uns.ac.rs; orcid: \href{https://orcid.org/0000-0002-2395-7628}{0000-0002-2395-7628})}
\thanks{This work was supported by the European Union's Horizon 2020
research and innovation programme under Grant Agreement number 856967,
and by the Secretariat for Higher Education and Scientific Research of
the Autonomous Province of Vojvodina through the project number 142-451-2686/2021.}

 %, 94B25.}

\date{April 23, 2022.}


\begin{abstract}
Two constructions of lattice packings of $ n $-dimensional cross-polytopes
($ \ell_1 $ balls) are described, the density of which exceeds that of any
prior construction by a factor of at least $ 2^{\frac{n}{\ln n}(1 + o(1))} $
when $ n \to \infty $.
The first family of lattices is explicit and is obtained by applying
Construction A to a class of Reed--Solomon codes.
The second family has subexponential construction complexity and is based
on the notion of Sidon sets in finite Abelian groups.
The construction based on Sidon sets also gives the highest known
asymptotic density of packing discrete cross-polytopes of fixed radius
$ r \geqslant 3 $ in $ \mathbb{Z}^n $.
\end{abstract}

\maketitle

\section{Introduction}
\label{sec:intro}

Dense packings of spheres and other bodies in Euclidean spaces have been
objects of mathematical research for centuries \cite{conway+sloane, gruber+lek, rogers},
and have also found applications in various fields such as coding theory
and physics.
In this paper we consider the problem of efficiently packing cross-polytopes
and give two simple constructions of \emph{lattice} packings in $ \mathbb{R}^n $,
for arbitrary $ n $, of density larger than that of any prior construction.

An $ n $-dimensional cross-polytope $ C_n $ is a unit ball in $ \mathbb{R}^n $
with respect to the $ \ell_1 $ metric,
$ C_n = \left\{ \bs{y} \in \mathbb{R}^n : \sum_{i=1}^n |y_i| \leqslant 1 \right\} $.
A cross-polytope of radius $ r \in \mathbb{R} $ is the body
$ r C_n = \{ r \bs{y} : \bs{y} \in C_n \} $ of volume $ \frac{(2r)^n}{n!} $.
By a discrete cross-polytope of radius $ r \in \mathbb{Z} $ we mean the
set $ (rC_n) \cap \mathbb{Z}^n $ of cardinality
$ \sum_{j \geqslant 0} 2^j \binom{n}{j} \binom{r}{j} $.

A lattice packing of cross-polytopes of radius $ r $ is an arrangement
of these bodies in $ \mathbb{R}^n $ of the form
$ {\mathcal L} + r C_n = \{ \bs{x} + \bs{y} : \bs{x} \in {\mathcal L}, \bs{y} \in r C_n \} $,
where $ \mathcal L $ is a lattice (a discrete additive subgroup of
$ \mathbb{R}^n $) of minimum $ \ell_1 $ distance $ \geqslant\! 2r $.
The density of such a packing is the fraction of space covered by the
cross-polytopes;
it can be computed as the ratio of the volume of a cross-polytope and
the determinant of the lattice $ \mathcal L $ (the volume of its
fundamental cell), that is $ \frac{(2r)^n}{n! \det{\mathcal L}} $.

Rush \cite{rush} gave a construction of lattice packings of cross-polytopes
in $ \mathbb{R}^n $, for $ n = \frac{p-1}{2} $, $ p $ an odd prime, of density
%\begin{subequations}
\begin{align}
\label{eq:density_rush}
  \max_{1 \leqslant t \leqslant n} \frac{ (2t + 1)^{n} }{ n! \, (2n + 1)^t } .
\end{align}
The value of $ t $ for which the maximum in \eqref{eq:density_rush} is attained
is \cite{rush}
\begin{align}
\label{eq:maxt}
  t = \frac{ n }{ \ln(2n + 1) } - \frac{1}{2} .
\end{align}
%\end{subequations}
In the present paper we describe two constructions of lattice packings in
arbitrary dimension, one based on Reed--Solomon codes (Section~\ref{sec:RS})
and the other based on the notion of Sidon sets in finite Abelian groups
(Section~\ref{sec:Sidon}), both of which exceed the density in
\eqref{eq:density_rush} by a factor that scales as
$ 2^{\frac{n}{\ln n}(1 + o(1))} $ when $ n \to \infty $.

As pointed out in \cite{rush}, packings of much higher density can be
shown to exist by non-constructive methods such as the Minkowski--Hlawka
(MH) theorem: the MH lower bound on the lattice packing density of
cross-polytopes is of the form $ 2^{-n + o(n)} $ \cite{rogers, rush3},
while \eqref{eq:density_rush} scales as $ e^{-n \ln\ln n + {\mathcal O}(n)} $.
The most efficient known algorithms for constructing lattices that achieve
the MH bound (up to lower order terms in the exponent) have complexity
$ 2^{{\mathcal O}(n \log n)} $ and are obtained by exhaustive search methods
such as the Gilbert--Varshamov bound from coding theory \cite{rush3}.%
\footnote{Regarding the MH theorem, we point the reader to the recent paper
\cite{gargava+serban} and references therein for (subexponential) improvements
of this bound in the sphere-packing case, and constructions of complexity
$ 2^{{\mathcal O}(n \log n)} $ that achieve it. We also note that, in the
case of packing superballs ($ \ell_\sigma $ balls), exponential improvements
of the MH bound are known when $ \sigma > 2 $ \cite{rush3, elkies}. In the
case of cross-polytopes ($ \sigma = 1 $), it is conjectured \cite{rush2}
that no such exponential improvement is possible.}
It is desirable, however, both from the mathematical viewpoint and in
applications, to have at one's disposal more {efficient} and {explicit}
constructions of packings.

How one defines constructiveness of a packing is to an extent subjective,
but one very natural definition \cite{litsyn} is that, for a given family
of lattice packings, there exists an algorithm for constructing the lattices
(e.g., for producing their basis vectors) whose complexity grows polynomially
with the dimension $ n $.
The first family of packings described in this paper is constructive in this
sense.
While it is not clear whether the second family can also be constructed in
polynomial time, we do show that it can be constructed by using probabilistic
algorithms of subexponential complexity $ 2^{{\mathcal O}(\sqrt{n \log n})} $.
Furthermore, the second family can also be considered constructive in the
(weak) sense that it ``arises from other natural mathematical objects'',
see the discussion by Litsyn and Tsfasman \cite[Section 3]{litsyn}.
Finally, as we point out in Section~\ref{sec:discrete}, the construction
based on Sidon sets is interesting for the following reason as well: when
the radius $ r $ is fixed, this construction is of polynomial complexity
and produces densest known packings of discrete cross-polytopes in
$ \mathbb{Z}^n $ in the asymptotic regime $ n \to \infty $, for any
$ r \geqslant 3 $.

To conclude the introductory part of the paper, let us mention that dense
packings of cross-polytopes also induce reasonably dense packings of superballs
($ \ell_\sigma $ balls) for small values of $ \sigma $ ($ 1 \leqslant \sigma < 2 $)
by using the trivial method of inscribing a superball inside a cross-polytope, see
\cite{rush}.

\section{Construction based on {R}eed--{S}olomon codes}
\label{sec:RS}

Given positive integers $ n, t $ and a prime $ p $ satisfying
$ 1 \leqslant t \leqslant n < p $, let $ {\mathcal C}_{n,t;p}^{\textsc{rs}} $
be the set of all vectors $ \bs{x} \in \mathbb{Z}_p^n $ satisfying the
congruences $ \sum_{i=1}^n i^s x_i = 0  \pmod {p} $, for $ s = 0, 1, \ldots, t-1 $.
The set $ {\mathcal C}_{n,t;p}^{\textsc{rs}} $ is a (generalized)
Reed--Solomon code of length $ n $ over the field $ GF(p) \equiv \mathbb{Z}_{p} $
\cite[Chapter 5]{roth}.
Both the cardinality, $ |{\mathcal C}_{n,t;p}^{\textsc{rs}}| =  p^{n-t} $,
and the minimum distance properties of the code are controlled by the
parameter $ t $.
Namely, the minimum Hamming distance
%\footnote{Hamming distance between two vectors is the number of coordinates
%at which they differ.}
of $ {\mathcal C}_{n,t;p}^{\textsc{rs}} $ is $ t + 1 $ \cite[Proposition 5.1]{roth},
while its minimum Lee distance%
\footnote{Lee distance is essentially the $ \ell_1 $ distance defined on
the torus $ \mathbb{Z}_m^n $. A code in $ \mathbb{Z}_m^n $ having minimum
Lee distance $ d $ can therefore be thought of as a packing of discrete
cross-polytopes ($ \ell_1 $ balls) of radius $ \lfloor \frac{d-1}{2} \rfloor $
in the torus, see \cite{golomb+welch}.}
is lower bounded by $ 2 t $ \cite[Theorem 3]{roth+siegel}. %\cite[Theorem 10.5]{roth}

Let
$ {\mathcal L}_{n,t;p}^{\textsc{rs}} = p \mathbb{Z}^n + {\mathcal C}_{n,t;p}^{\textsc{rs}} $
be the lattice obtained by periodically extending the above code to all of
$ \mathbb{Z}^n $.
Written explicitly,
\begin{align}
\label{eq:latticeRS}
  {\mathcal L}_{n,t;p}^{\textsc{rs}}
   = \left\{ \bs{x} \in \mathbb{Z}^n \;:\; \sum_{i=1}^n i^s x_i = 0  \pmod {p}, \quad
	           s = 0, 1, \ldots, t-1 \right\} .
\end{align}
The fact that the minimum Lee distance of the code
$ {\mathcal C}_{n,t;p}^{\textsc{rs}} $ is at least $ 2 t $ implies that
the minimum $ \ell_1 $ distance of the lattice $ {\mathcal L}_{n,t;p}^{\textsc{rs}} $
is at least $ 2 t $.
This lattice therefore induces a packing of cross-polytopes of radius $ t $.
Since $ \det {\mathcal L}_{n,t;p}^{\textsc{rs}} = p^t $, the density of
the packing is $ \frac{(2t)^n}{n!\,p^t} $.
Moreover, the lattice can be constructed efficiently, which is evident from
its definition.
We have just shown the following.

\begin{theorem}
\label{thm:RS}
For every $ n \geqslant 1 $, the cross-polytope can be constructively lattice
packed in $ \mathbb{R}^n $ with density
\begin{align}
\label{eq:densityRS}
  \max_{1 \leqslant t \leqslant n} \frac{ (2 t)^{n} }{ n! \operatorname{p}(n)^{t} } ,
\end{align}
where $ \prost(n) $ is the smallest prime larger than $ n $.
\end{theorem}

The value of $ t $ that maximizes the expression in \eqref{eq:densityRS} also
maximizes $ n \ln(2 t) - t \ln\prost(n) $, and by differentiating the latter
we find this value to be
\begin{align}
\label{eq:maxtRS}
  t = \frac{n}{\ln\prost(n)} .
\end{align}

The same method of constructing lattices from codes was used in \cite{rush}
(the so-called Construction A from \cite[Chapter 5]{conway+sloane}), but
the starting point therein was the Berlekamp's negacyclic code \cite[Chapter 10.6]{roth}
of length $ n = \frac{p-1}{2} $, where $ p $ is an odd prime, and minimum
Lee distance $ 2t + 1 $.

The main advantages of the construction \eqref{eq:latticeRS} with respect
to \cite{rush} are the following:
\begin{inparaenum}
\item[(a)]
the packing \eqref{eq:latticeRS} is defined in all dimensions,
\item[(b)]
the construction \eqref{eq:latticeRS} is explicit, while that from \cite{rush}
requires finding a primitive element in the field $ GF(p) $, and
\item[(c)]
the packing density of \eqref{eq:latticeRS} is larger by a factor that
scales as $ 2^{\frac{n}{\ln n}(1 + o(1))} $ when $ n \to \infty $.
\end{inparaenum}
To justify (c), suppose that $ n = \frac{p-1}{2} $ for an odd prime $ p $,
and note that the ratio of the densities from \eqref{eq:densityRS} and
\eqref{eq:density_rush} equals $ ( \frac{2t}{2t + 1} )^n ( \frac{2n+1}{\prost(n)} )^t $,
where $ t \sim \frac{n}{\ln n} $ (see \eqref{eq:maxt} and \eqref{eq:maxtRS}).
Recalling that $ \prost(n) = n + o(n) $ for sufficiently large $ n $ \cite[Section 1.4.1]{primes},
our claim follows.

\section{Construction based on {S}idon sets}
\label{sec:Sidon}

A collection of elements $ b_1, b_2, \ldots, b_n $ from an Abelian group
$ (G, +) $ having the property that the sums $ b_{i_1} + b_{i_2} + \cdots + b_{i_h} $,
$ 1 \leqslant i_1 \leqslant i_2 \leqslant \cdots \leqslant i_h \leqslant n $,
are all different is called a Sidon set of order $ h $.%
\footnote{Or a Sidon sequence of order $ h $, or a $ B_h $ sequence.
These objects have been studied quite extensively \cite{obryant}.
For more on their connection to lattice packing problems,
see \cite{kovacevic+tan_sidma, kovacevic+tan}.}
An equivalent way of expressing this property is that the sums
\begin{align}
\label{eq:sidon}
  \sum_{i=1}^n r_i b_i , \quad  \textnormal{where}\;\;  r_i \in \mathbb{Z} , \;
	r_i \geqslant 0 , \;  \sum_{i=1}^n r_i = h , \quad \textnormal{are all different} .
\end{align}
Here $ r_i b_i $ represents the sum of $ r_i $ copies of the element $ b_i \in  G $.
Two elegant constructions of Sidon sets were described by Bose and Chowla in
\cite{bose+chowla}, one of which is repeated next for completeness.

For a prime power $ n $, let $ \alpha_1 = 0, \alpha_2, \ldots, \alpha_n $ be
the elements of the Galois field $ GF(n) $, and $ \beta $ a primitive element
of the extended field $ GF(n^{h}) $ (i.e., a generator of its multiplicative
group).
Let $ b_1, b_2, \ldots, b_n $ be the numbers from the set $ \{1, 2, \ldots, n^{h}-1\} $
defined by
\begin{align}
\label{eq:bi}
  \beta^{b_i} = \beta + \alpha_i ,  \quad i = 1, \ldots, n .
\end{align}
Then the numbers $ b_1 = 1, b_2, \ldots, b_n $, thought of as elements of
the cyclic group $ (\mathbb{Z}_{n^h - 1}, +) $, satisfy the condition \eqref{eq:sidon}.
To see that they do, suppose, for the sake of contradiction, that
$ b_{i_1} + b_{i_2} + \cdots + b_{i_h} = b_{j_1} + b_{j_2} + \cdots + b_{j_h} $
for two different sets of indices $ \{i_k\} $, $ \{j_k\} $.
Then it would follow from \eqref{eq:bi} that
\begin{align}
\label{eq:bi2}
  (\beta + \alpha_{i_1})(\beta + \alpha_{i_2}) \cdots (\beta + \alpha_{i_h})
  = (\beta + \alpha_{j_1})(\beta + \alpha_{j_2}) \cdots (\beta + \alpha_{j_h})
\end{align}
and, after canceling the $ \beta^h $ terms, that $ \beta $ is a root of
a polynomial of degree $ <\! h $ with coefficients in $ GF(n) $, which is
not possible.
When the desired cardinality $ n $ is not a prime power, one can use the
same method to produce a Sidon set $ b_1, b_2, \ldots, b_{\pprost(n)} $,
where $ \pprost(n) $ is the smallest prime power greater than or equal to
$ n $, and keep any $ n $ of its $ \pprost(n) $ elements (a subset of a
Sidon set is also a Sidon set).
%Alternatively, one can use different constructions of Sidon sets...
%For example, the second construction from \cite{bose+chowla} produces a Sidon
%set of size $ n $ in the group $ \mathbb{Z}_{(n-1)^t + (n-1)^{t-1} + \cdots + 1} $,
%when $ n - 1 $ is a prime power.

Let $ g^{\textsc{s}}(h, n) $ denote the size of the smallest Abelian group
containing a Sidon set of order $ h $ and cardinality $ n $.
From the Bose--Chowla construction just described we know that
$ g^{\textsc{s}}(h, n) < \pprost(n)^{h} $.

\begin{theorem}
\label{thm:Sidon}
For every $ n \geqslant 1 $, the cross-polytope can be constructively
(in the weak sense) lattice packed in $ \mathbb{R}^n $ with density
\begin{align}
\label{eq:density}
	\max_{t \geqslant 1} \frac{ (2t)^{n-1} }{ n! \, g^{\textsc{s}}(t-1, n) } \; > \;
  \max_{t \geqslant 1} \frac{ (2t)^{n-1} }{ n! \pprost(n)^{t-1} } ,
\end{align}
where $ \pprost(n) $ is the smallest prime power greater than or equal to $ n $.
\end{theorem}

The value of $ t $ that maximizes the expression on the right-hand side
of the inequality \eqref{eq:density} is
\begin{align}
\label{eq:mymaxt}
  t = \frac{n-1}{\ln\pprost(n)} .
\end{align}

\begin{proof}[Proof of Theorem \ref{thm:Sidon}]
Given a Sidon set $ B = \{ b_1, b_2, \ldots, b_n \} $ of order $ t - 1 $ in
an Abelian group $ G $, define the following lattice:
\begin{align}
\label{eq:L}
  {\mathcal L}_{n,t;B}^{\textsc{s}}
   = \left\{ \bs{x} \in \mathbb{Z}^n \;:\; \sum_{i=1}^n x_i = 0  \pmod {2t}, \quad
	           \sum_{i=1}^n x_i b_i = 0 \right\} .
\end{align}
Here $ x_i b_i $ represents the sum of $ |x_i| $ copies of the element $ b_i \in  G $
(resp.\ $ -b_i \in  G $) if $ x_i \geqslant 0 $ (resp.\ $ x_i < 0 $).
We claim that the minimum $ \ell_1 $ distance of the points in this lattice
is $ 2 t $.
To see this, note that any two points $ \bs{x}, \bs{y} \in {\mathcal L}_{n,t;B}^{\textsc{s}} $
with $ \sum_{i=1}^n x_i  \neq  \sum_{i=1}^n y_i $ satisfy
$ \sum_{i=1}^n (x_i - y_i) = 2 t k $ for a nonzero integer $ k $.
They must be at distance at least $ 2 t $ because
$ \sum_{i=1}^n |x_i - y_i| \geqslant \left| \sum_{i=1}^n (x_i - y_i) \right| = 2 t |k| $.
Now consider $ \bs{x}, \bs{y} \in {\mathcal L}_{n,t;B}^{\textsc{s}} $,
$ \bs{x} \neq \bs{y} $, with $ \sum_{i=1}^n x_i  =  \sum_{i=1}^n y_i $.
Suppose that, for two such points, $ \sum_{i=1}^n |x_i - y_i| \leqslant 2 (t - 1) $
(the distance is in this case necessarily even).
Then one can write $ \bs{x} + \bs{r} = \bs{y} + \bs{s} $, for some
$ \bs{r}, \bs{s} \in \mathbb{Z}^n $, $ \bs{r} \neq \bs{s} $, with
$ r_i \geqslant 0 $, $ s_i \geqslant 0 $, and
$ \sum_{i=1}^n r_i = \sum_{i=1}^n s_i = t - 1 $.
This, together with the fact that $ \sum_{i=1}^n x_i b_i = \sum_{i=1}^n y_i b_i = 0 $
(see~\eqref{eq:L}), implies $ \sum_{i=1}^n r_i b_i = \sum_{i=1}^n s_i b_i $.
As this contradicts \eqref{eq:sidon} (with $ h = t - 1 $), our assumption
that $ \sum_{i=1}^n |x_i - y_i| \leqslant 2 (t - 1) $ must be wrong.
Therefore, as claimed, the minimum $ \ell_1 $ distance of the lattice
$ {\mathcal L}_{n,t;B}^{\textsc{s}} $ is $ 2 t $, implying that it induces
a packing of cross-polytopes of radius $ t $.
Since $ \det {\mathcal L}_{n,t;B}^{\textsc{s}} = 2 t |G| $, the obtained
packing density is $ \frac{(2t)^{n-1}}{n! \, |G|} $.
Furthermore, by the result from \cite{bose+chowla} cited above we may take
$ G = \mathbb{Z}_{\pprost(n)^{t-1} - 1} $, implying the lower bound in
\eqref{eq:density}.
\end{proof}

The density obtained in Theorem~\ref{thm:Sidon} is comparable to that from
Theorem~\ref{thm:RS}.
The former is also larger than the density obtained in \cite{rush} by a
factor that scales as $ 2^{\frac{n}{\ln n}(1 + o(1))} $ when $ n \to \infty $,
see \eqref{eq:density_rush} and \eqref{eq:density}.
An additional advantage of the construction based on Sidon sets, compared
to both Theorem~\ref{thm:RS} and \cite{rush}, is that the packing is defined
for every $ t $, i.e., the dimension and the radius are independent
variables in this approach.

It should be noted, however, that the construction complexity of the family
of lattices \eqref{eq:L} is higher.
In particular, constructing a Sidon set of cardinality $ n $ and order $ t-1 $
that was described in \cite{bose+chowla} (see the second paragraph of this section)
involves:
\begin{inparaenum}
\item[(i)]
finding a primitive element $ \beta $ in the field $ GF(\pprost(n)^{t-1}) $, and
\item[(ii)]
finding solutions in $ \{1, 2, \ldots, \pprost(n)^{t-1} - 1\} $ to $ n $ equations
of the form \eqref{eq:bi}.
\end{inparaenum}
Note that, when $ t \sim \frac{n}{\ln n} $ (see \eqref{eq:mymaxt}), the required
field size is exponential in the dimension, $ \pprost(n)^{t-1} = e^{n+o(n)} $.
The problem (ii) is an instance of the discrete logarithm problem which, as recent
advances have shown \cite{granger, kleinjung}, can be solved in expected
quasi-polynomial time $ 2^{{\mathcal O}(\log^2 n)} $.
As for problem (i) -- finding a primitive element in $ GF(\pprost(n)^{t-1}) $,
$ t \sim \frac{n}{\ln n} $ -- a result of Shoup \cite{shoup} implies that it
can be reduced in time polynomial in $ n $ to the problem of (i')~\emph{testing}
whether a given element $ \beta $ is primitive.
A classical way of solving (i') is by factoring $ \pprost(n)^{t-1} - 1 $ (the
order of the multiplicative group of $ GF(\pprost(n)^{t-1}) $) and checking
whether $ \beta^x = 1 $ for any non-trivial factor $ x $.
Factoring numbers of this magnitude can be performed in expected time%
\footnote{Faster methods are known, such as the number field sieve, but they
are not rigorous and rely on heuristics \cite[Chapter 6]{primes}.}
$ 2^{{\mathcal O}(\sqrt{n \log n})} $ \cite[Chapter 6]{primes}.
In conclusion, the complexity of constructing \eqref{eq:L} is dominated by the
problem of finding primitive elements in large finite fields, and can be upper
bounded by $ 2^{{\mathcal O}(\sqrt{n \log n})} $.

\section{Lattice packings of discrete cross-polytopes in $ \mathbb{Z}^n $}
\label{sec:discrete}

A lattice packing of discrete cross-polytopes of radius $ t \in \{1, 2, \ldots\} $
is an arrangement of discrete cross-polytopes in $ \mathbb{Z}^n $ of the
form $ {\mathcal L} + (t C_n \cap \mathbb{Z}^n) $, where $ \mathcal L $ is
a sublattice of $ \mathbb{Z}^n $ with minimum $ \ell_1 $ distance $ >\! 2t $
(note the strict inequality here).
The density of such a packing -- the fraction of points in $ \mathbb{Z}^n $
covered by the cross-polytopes -- is
$ \frac{ |(t C_n) \cap \mathbb{Z}^n| }{ \det \mathcal{L} } $.

Apart from being interesting on their own, packings in $ \mathbb{Z}^n $ are
useful for producing packings in $ \mathbb{R}^n $, as we have seen in the
previous two sections.
In fact, most lattices described in the literature for the purpose of packing
various convex bodies in $ \mathbb{R}^n $, are sublattices of $ \mathbb{Z}^n $.
Moreover, one can show that optimal lattice packings of, e.g., cross-polytopes
in $ \mathbb{R}^n $, for any fixed $ n $, can be obtained via optimal lattice
packings of discrete cross-polytopes of radius $ t \to \infty $ in $ \mathbb{Z}^n $
\cite[Remark 2.2]{kovacevic+tan_sidma}.
Discrete packings are also of interest in coding theory where they frequently
represent the underlying geometric problem.
For example, an appropriate finite restriction of a lattice
$ {\mathcal L} \subseteq \mathbb{Z}^n $ with minimum $ \ell_1 $ distance
$ \geqslant\!2t+1 $ can be interpreted as a code correcting $ t $ errors
of certain type \cite{kovacevic+tan}.

We state below the discrete version of Theorem~\ref{thm:Sidon}, as it may
be of separate interest.
The construction based on Sidon sets appears to produce very dense packings
of discrete cross-polytopes, at least in the case when the radius $ t $ is
fixed and $ n \to \infty $.
Furthermore, in this regime the construction is of polynomial complexity.

\begin{theorem}
Fix an arbitrary positive integer $ t $.
For every $ n \geqslant 1 $, the discrete cross-polytope of radius $ t $
can be constructively lattice packed in $ \mathbb{Z}^n $ with density
\begin{align}
\label{eq:densityZ}
	\frac{ \sum_{j \geqslant 0} 2^j \binom{n}{j} \binom{t}{j} }{ (2t+1) \, g^{\textsc{s}}(t, n) } \; > \;
  \frac{ \sum_{j \geqslant 0} 2^j \binom{n}{j} \binom{t}{j} }{ (2t+1) \pprost(n)^{t} } ,
\end{align}
where $ \pprost(n) $ is the smallest prime power greater than or equal to $ n $.
\end{theorem}
\begin{proof}
A lattice of minimum $ \ell_1 $ distance $ \geqslant\! 2t+1 $ can be obtained
by using the same construction as in \eqref{eq:L}, with two minor modifications:
we require that $ \sum_{i=1}^n x_i = 0  \pmod{2t + 1} $ (instead of $ \mod {2t} $),
and that $ B $ is a Sidon set of order $ t $ (instead of $ t - 1 $).
The lattice can be efficiently constructed because the radius is fixed and hence
the group containing a Sidon set is of polynomial size, $ {\mathcal O}(n^t) $.
\end{proof}

For a fixed radius $ t $ and $ n \to \infty $, the asymptotic value of
the expression on the right-hand side of \eqref{eq:densityZ} is
\begin{align}
\label{eq:asymp_dens}
  \frac{2^t}{t! \, (2t+1)} .
\end{align}
For $ t = 1, 2 $, this lower bound can be improved.
For $ t = 1 $, the maximum possible density of $ 1 $ can be achieved fore
every $ n $, as perfect packings of discrete cross-polytopes of radius $ 1 $
exist (and are easily constructed) in all dimensions \cite{golomb+welch}.
For $ t = 2 $, the construction from \cite{rush} yields the asymptotic
density $ \frac{1}{t!} = \frac{1}{2} $, while the expression in
\eqref{eq:asymp_dens} equals $ \frac{2}{5} $.
For $ t \geqslant 3 $, the asymptotic density in \eqref{eq:asymp_dens}
is, to the best of our knowledge, the highest known.

\section*{Acknowledgements}

The author would like to thank the referees for their thorough reading and
\emph{constructive} comments which greatly improved the manuscript, and
Benjamin Wesolowski for clarifying several points about \cite{kleinjung}.

%\vfill

\end{document}